\documentclass[12pt,reqno]{amsart}

\usepackage{amsfonts,amssymb}
\usepackage{times,amsmath,amstext,amsbsy,amsopn,amsthm,upref,amscd,eucal}
\usepackage[all]{xy}
\usepackage{enumerate}

\allowdisplaybreaks

\newtheorem{theorem}{Theorem}[section]
\newtheorem{Luna Slice Theorem}[theorem]{Luna Slice Theorem}

\newtheorem{corollary}[theorem]{Corollary}
\newtheorem{lemma}[theorem]{Lemma}
\newtheorem{proposition}[theorem]{Proposition}
\newtheorem{proposition-definition}[theorem]{Proposition-Definition}

\theoremstyle{definition}
\newtheorem{definition}[theorem]{Definition}

\theoremstyle{remark}
\newtheorem{remark}[theorem]{Remark}

\newcommand{\fatdot}{{\scriptscriptstyle \bullet}}

\newcommand{\rk}{\operatorname{rk}\nolimits}

\newcommand\CC{{\mathbb C}}

\newcommand\QQ{{\mathbb Q}}

\newcommand\ZZ{{\mathbb Z}}

\newcommand\HH{{\mathbb H}}

\newcommand{\BBB}{{\mathcal B}}
\newcommand{\GGG}{{\mathcal G}}
\newcommand{\UUU}{{\mathcal U}}
\newcommand{\LLL}{{\mathcal L}}

\newcommand{\FFF}{{\mathcal F}}
\newcommand{\KKK}{{\mathcal K}}

\newcommand{\MMM}{{\mathcal M}}
\newcommand{\CCC}{{\mathcal C}}

\newcommand{\OOO}{{\mathcal O}}
\newcommand{\EEE}{{\mathcal E}}

\newcommand{\HHH}{{\mathcal H}}
\renewcommand\phi{\varphi}

\newcommand{\Ad}{\operatorname{Ad}\nolimits}

\newcommand{\ad}{\operatorname{ad}\nolimits}
\newcommand{\Mor}{\operatorname{Mor}\nolimits}

\newcommand{\id}{\operatorname{id}\nolimits}

\newcommand{\Res}{\operatorname{Res}\nolimits}

\newcommand{\Vect}{\operatorname{Vect}\nolimits}

\newcommand{\Tr}{\operatorname{Tr}\nolimits}

\newcommand{\Quot}{\operatorname{Quot}\nolimits}

\newcommand\lra{{\longrightarrow}}

\newcommand\ra{{\rightarrow}}
\newcommand\rar{{\rightarrow}}

\begin{document}
\title{Gromov Witten invariants for maximal parabolic vector bundles over an orbifold}
\author{Francois-Xavier Machu}
\address{F-X. M.: Mathematics Department, EISTI-University Paris-Seine,Avenue du parc, 95000, Cergy-Pontoise.}
\email{fmu@eisti.eu}

\bigskip
\bigskip
\begin{abstract}
We define the Gromov-Witten invariants for the parabolic bundles over an orbifold $C$ in various
situation. Those bring us to refine this notion to get an accurate computation
of the number of maximal subbundles of a sufficiently general parabolic bundle by means of the Intriligator-Vafa formula.
\end{abstract}
\maketitle
\vspace{1 ex}
\begin{center}\begin{minipage}{110mm}\footnotesize{\bf Key words:} Gromov-Witten invariants, parabolic vector bundles, twisted map, stack. 
\end{minipage}
\end{center}

\begin{center}\begin{minipage}{110mm}\footnotesize{\bf MSC2000:} 14B12, 14F05, 14F40, 14H60, 32G08.

\end{minipage}
\end{center}
\vspace{1 ex}

\section*{Introduction}
The Gromov-Witten invariants of the Grassmannian that 
parameterizes the $k$-dimensional vector subspaces of the restriction of a vector bundle
$E$ of rank $r$ and of degree $d$ on a smooth complex projective curve $C$ of genus
$g(C)=g\geq 2$ over a point $x$ of $C$ have been computed \cite{Ho-2} as a an intersection number in the Quot-scheme.
These intersection numbers are defined by means of the Schubert schemes that are the
degeneracy loci of vector bundles constructed from the property of the universality of the
Grassmannian projective scheme and with their correspondence Chern class by intersection theory.
From this, the Gromov-Witten invariants are defined for a vector bundle $E$ and weighted
homogeneous polynomials $P$ with extra hypotheses as  
$P(c_1(F^*),...,c_k(F^*))\cap[Quot^{k,e}(E)]$, where $Quot$ parameterizes the surjections $E\ra G$, with
$G$ is a locally free sheaf of rank $k$ and degree $e$ + particular condition on $e$.
The dual of $F$, where $F$ is defined this time by the universal property of the
Quot-scheme. These numbers match with the Gromov-Witten invariants  defined in the general
framework for the moduli space of stable maps $f$ from a $n$-pointed connected nodal curve $C$ to
the Grassmannian whose $f_*([C])=\beta$.

On the other hand, the number $m(r,d,k,g)$ is the number of maximal
subbundles of a general stable bundle $E$, where $k$ is defined by the $s_k$-invariant
of $E$, $s_k(E):=dk-re$, where $e=e_{max,d}(E)$ with $e=max\{deg(F),\rk(F)=k\}$.
This number is none than the number points in the Quot-scheme that is can be simply
defined by the weighted homogeneous polynomial.

These numbers are computed by the Intriligator-Vafa formula when the considered
morphism is of integer degree.

Our principal goal is to compute the Gromov-Witten invariants in the parabolic case.
To realize this, we start by adding an extra structure, a parabolic structure
on vector bundles (this implies that the morphism can have a rational degree). In this case, 
the following changes perform as follows. We replace the nodal curve by an orbifoldcurve, the
Deligne-Mumford moduli stack of stable maps by the Kontsevich moduli stack of twisted stables
maps with their relevant evaluation maps to rigified inertia stack. 
Note that the parabolic structure is over marking and node points have a stacky structure (root stack structure).
In Sect $1$, we construct the Quot-stack. In Sect $2$, we define the global Gromov-Witten invariants for parabolic case.
In Sect $3$, those are defined locally by means of the deformation theory. 
In Sect $4$, these latter are defined in a general manner and show that they do not depend on the
choice of the orbifold curve.
In Sect $5$, we establish some properties in the case of general parabolic stable bundles and
establish the Intriligator-Vafa formula. In last Sect, we give some examples of computation in the
case of finding the number $m(n,d,k,g)$ of maximal subbundles of a sufficiently parabolic stable bundle.

\section{Construction of the Quot-stack}
We start with a twisted curve that is a gadget $\Sigma_i\subset\CCC\ra C$, where
$\CCC$ is a Deligne-Mumford stack over a coarse moduli connected nodal curve $C$ such that:\\
Over a node $\{xy=0\}$, its local chart is $[\{uv=0\}/\mu_r]$ given in local
coordinates by $\CCC\ra C, (u,v)\mapsto (x=u^r,y=u^r)$.\\
Over a marking point, its local chart is $[A^1/\mu_r]$ given in local
coordinates by $\CCC\ra C, u\mapsto x=u^r$.\\
The $\Sigma_i$ are defined locally by $\{u=0\}$, are gerbes marked by $\mu_r$.

\begin{definition}
Consider a scheme $X$, a line bundle $\LLL$ on $X$, a section $s\in\Gamma(X,L)$, 
and a positive integer $r$. Define a root stack 
$$r\sqrt{(L/X,s)},$$
whose objects over a scheme $Y$ are the triples $(f:Y\ra X,M,t),$
where $M$ is a line bundle on $Y$ with section $t$ such that
$M^{\otimes r}\simeq f^*(L), t^r=f^*(s).$\\
For a Cartier divisor $D$, we adopt the Vistoli's notation
$$r\sqrt{(X,D)}:=r\sqrt{(\OOO_X(D),\id_D)}.$$
\end{definition}
\begin{remark}
This stack is isomorphic to $X$ away from the zero divisor $D$
of the section, and canonically introduces a stack structure with
index $r$ along $D$, is said minimal if $D$ is smooth. This immediately
enables us to define the stacky structure of a twisted curve at marking starting
with the coarse moduli curve:
$$(C,p)\hookrightarrow\CCC=r\sqrt{(C,p)}.$$
\end{remark}
To deal with the case of the nodes, we will suppose that
the nodes are separating to use root stacks merely, otherwise
one needs either subtle descent or logarithmic structures.

We assume that in our case we have a $n$-pointed twisted curve $\CCC$ and
the preimage of $D$ separated by $Z\subset C$ the locus of nodes consists of two
connected components $E_1$ and $E_2$.
Olson established the existence of a universal stack $\CCC^{tw}_{g,n}$ of $n$-pointed twisted curves of genus $g$
over the stack of twisted curves $\MMM^{tw}_{g,n}=\MMM$. We have the structure morphism
$V\ra\MMM^{tw}_{g,n}$, where $V$ is a polydisk. Denote $\MMM^{tw}_{r}$, the locus where the given node
is given stacky structure of index $r$, and $\CCC^{tw}_r$ the universal twisted curve. Then we have
$$V\times_{\MMM}\MMM^{tw}_r=r\sqrt{(V,D)},$$ and $$V\times_{\MMM}\CCC^{tw}_r=r\sqrt{(C,E_1)}\times r\sqrt{(C,E_2)}.$$
Since $\MMM^{tw}_{r}\ra\MMM$ is birational, but the versal deformation of nodal curves is branched with index $r$
over $D$, this branching is accounted for by automorphisms of the twisted curve. We deduce the
automorphism group of a twisted curve fixing $C$ is
$$Aut(\CCC)=\prod_{s\in Sing(\CCC)} \Gamma_s,$$ where $\Gamma_s\simeq\mu_{rs}$ is the stabilizer
of the corresponding node.

Let $\chi$ be a Deligne-Mumford stack over a coarse projective scheme $X$.
The parabolic structure on a vector bundle over a scheme is parameterized by a product of flag schemes.
Moreover we have an equivalence of tensor categories between the category of vector bundles on a stack $\chi$ and the
category of parabolic bundles on a scheme $X$ after Th\'eor\`eme $3.13$ of \cite{Bo-1}. 
Therefore, it suggests that the parabolic structure is obtained throughout a
product of stacks. Hence, we investigate this notion.

We choose $\chi=Grass(\HHH=r\sqrt{(C,p_1)}\times...r\sqrt{(C,p_n)})$ that parameterizes
the morphisms of stacks of quotient modules whose objects over a scheme $T$
of finite type are the cartesian diagrams with relations
$$M_1^{\otimes r}\oplus...\oplus M_n^{\otimes r}\simeq f^*(\OOO_C(p_1)\oplus...\oplus\OOO_C(p_n)), \sum_{i=0}^{n} t_i^r=\sum_{i=0}^{n}f^*(s_i).$$
For consequently, there exists a Grassmannian stack of root stacks with a stacky structure at the nodes 
$\chi_n=Grass(r\sqrt{(C,E_1)}\times r\sqrt{(C,E_2)})$.
We define the functor  $$\boldsymbol{\Quot_{\CCC/C}((r\sqrt{(C,p_1)}\times...r\sqrt{(C,p_n)}),P)}:\CCC^{0}\rar \textrm{ Groupoids}$$ as follows:
If $T$ is a $k$-scheme, is associated the groupoid of the cartesian diagrams where $\FFF_1\oplus...\oplus\FFF_n$ is a quotient of the direct sum
of $\OOO_C(p_i)$ flat on $T$ whose fibers over the geometric points of the $S$-Grassmannian projective scheme have
Hilbert polynomial $P_i$ and $\GGG_1\oplus...\oplus\GGG_n$ is a quotient of the direct sum of $M_i$ whose fibers over the geometric 
points of the $S$-Grassmannian projective scheme have Hilbert polynomial $P'i$ with $P'i^r=P_i$.  
The Quot-functor parameterizes the set of $T$-flat coherent quotient stacks $\FFF$ of $\HHH_T$ such that
the fiber stacks over the geometric points of the $S$-Grassmannian stack have Hilbert polynomial $P$.
\begin{theorem}
The functor $$\boldsymbol{\Quot_{\CCC/C}((r\sqrt{(C,p_1)}\times...r\sqrt{(C,p_n)}),P)}$$ defined above is represented by a projective $S$-stack $\Quot_{\CCC/C}((r\sqrt{(C,p_1)}\times...r\sqrt{(C,p_n)}),P)$ with the
universal quotient stack $\UUU$.
\end{theorem}
\begin{proof}
See Theorem $2.2.4$ of \cite{H-L} in adapting to the case of stacks.
More precisely, there exits a representable stack $\Quot_{\CCC/C}((r\sqrt{(C,p_1)}\times...r\sqrt{(C,p_n)}),P)$ such that
$$\boldsymbol{\Quot_{\CCC/C}((r\sqrt{(C,p_1)}\times...r\sqrt{(C,p_n)}),P)}(S)\simeq\Mor(S,\Quot_{\CCC/C}((r\sqrt{(C,p_1)}\times...r\sqrt{(C,p_n)}),P).$$ Those relies on the fact the Quot-stack is a stratum in the stratification of the existence of the Grassmannian stack.

\end{proof}
In the same way, we also define the functor
 $$\boldsymbol{\Quot_{\CCC/C}(\BBB=(r\sqrt{(C,E_1)}\times\sqrt{(C,E_2)}),P)}:\CCC^{0}\rar \textrm{ Groupoids}$$
that is represented by a quasi-projective $S$-stack with the universal quotient stack $\UUU$.

We can also consider the stack $Grass(\HHH)$ parameterizing the quotient stacks of $\HHH$.
Then there is a universal quotient stack $\oplus\KKK_i\ra\HHH\otimes[\OOO_{\HHH}]$ with a natural action
of the product of the classifying stack $BGL(\HHH)$. Let us consider $H=\oplus H_i$ is contained
in the dual of $\HHH$, and $Y_H$ the Schubert stack defined as the degeneracy locus
of $\oplus H_i\otimes[\OOO_{\HHH}]\ra(\oplus K_i)^{\wedge}$.
Moreover $Y_H$ decomposes into integral stacks of codimension $i$.

\begin{proposition}
The functor $$\boldsymbol{\Mor_e(\CCC,Grass(\HHH))}:\CCC^{0}\rar \textrm{ Groupoids},$$
is defined as follows: If $T$ is a scheme over $S$, is associated the groupoid of morphisms
of degree $e$ from $T\otimes\CCC\ra Grass(\HHH)$. The latter is represented by a $S$ quasi-projective 
stack denoted $Mor_e(\CCC,Grass(\HHH))$.
The functor $$\boldsymbol{\Mor_f(\CCC,Grass(\BBB))}:\CCC^{0}\rar \textrm{ Groupoids},$$
is defined as follows: If $T$ is a scheme over $S$, is associated the the groupoid of morphisms
of degree $f$ from $T\otimes\CCC\ra Grass(\BBB)$. The latter is represented by a $S$-projective 
stack denoted $Mor_f(\CCC,Grass(\HHH)$.
\end{proposition}
\begin{proof}
We apply Theorem $2.2.4$ of \cite{H-L} in adapting to the case of the Grassmannian stack.
\end{proof}

\section{Gromov-Witten invariants}
We at present follow the steps of the paper \cite{Ho-2} to define the Gromov-Witten invariants for
the parabolic case.

We define the Gromov Witten invariants of the Grassmannian stack as intersection number
in the Quot-stack.
\begin{lemma}
Given a product of root stacks defined as previously, then there are some integers $n$ and $n'$ such that for each
$e'$ with $s_e'\geq n(\HHH)$, and $f'$ with $s_f'\geq n(\BBB)$, every component of the Quot-stack is generically smooth of expected dimension and a general element in every component corresponds to a substack of the product
of root stacks.
\end{lemma}
\begin{proof}
This is proved in \cite{P-R}, Theorem $6.4$. 
\end{proof}
Let $(p_1,...,p_n,E_1,E_2)\in C^{n+2}$ and substack $V_e'((p_1,...,p_n),\oplus H_i)\subset Quot^{k,e'}(\HHH)$ be
defined as the degeneracy locus of 
$$\oplus H_i\otimes\OOO_{Quot(\HHH)}\ra pr_1\check{\HHH}\oplus...\oplus pr_n\check{\HHH}\ra\check{F_1}\oplus...\oplus\check{F_n}=\check{\FFF}.$$
We define in the same way in the case of the nodes for a substack $V_f'((E_1,E_2),\oplus H'_i)\subset Quot^{k,f'}(\BBB)$.

Denote $s_e'=(n\sum_{i=1}^{n}k_i-ne')$, $s_f'=2(k_1+k_2)-2f'$, where $e'_n=max(\deg(F)$, with $F$ a subbundle of rank $k_n$ of $\HHH$.
Let $X_1,...,X_k$ be weighted variables such that the weight of $X_i$ is $i$. Let $P(X_1,...,X_{k'_{n}})$ with $k'_n=\sum_{1}^{n}k_i$  be a weighted homogeneous polynomial of weighted degree $s_e'+\sum_{i=1}^{n}k_i(n-\sum_{i=1}^{n}k_i)(1-g)$ (resp. $s_f'+(k_1+k_2)(2-(k_1+k_2))(1-g)$ with $s_e'>s(\HHH)(resp. s_f'>s(\BBB)$).
\begin{definition}
For a pair $(\HHH,e')$ over $C$ and $(\BBB,f')$, we define the Gromov-Witten invariants $N_{n,e'}P(X_1,...,X_{k_{n}})$ as
$$N_{n,e'}P(X_1,...,X_{k'_{n}})=\sum_{k_1,...,k_n}P(c_1(\check{\FFF}),...,c_{k_1+...+k_n}(\check{\FFF}))\cap[Quot_{vir}^{(k'_n,e')}(\HHH)],$$
and $$N_{n',f'}P(X_1,...,X_{k'_{2}})=\sum_{k_1,...,k_2}P(c_1(\check{\FFF}),...,c_{k_1+k_2}(\check{\FFF}))\cap[Quot_{vir}^{(k'_2),f')}(\BBB)].$$
\end{definition}

We specify that the fundamental cycle arising from the Gromov-Witten invariants of the Grassmannian stack as intersection number
in the Quot-stack whose geometric points are in one-to-one correspondences with  the fundamental classes $f_*[\CCC]=\beta _e\in H^2(Grass(\HHH,\beta_e),\QQ)$ of the maps arising from the Kontsevich Deligne Mumford stack $$\KKK_{g,n}(Grass(\HHH,\beta_e))$$ that consists of twisted stable maps $f$ from $\CCC$ to $Grass(\HHH,\beta_e)$. This latter will be proved thereafter.

We note that our $\Quot_{X/S}(r\sqrt{(C,p_1)}\times...r\sqrt{(C,p_n)})$ can be viewed as an augmented simplicial stack.
Moreover, a parabolic bundle over a scheme can be viewed as a functor $E_*:\frac{1}{r}\ZZ\ra\Vect(X)$ whose its 
degree is given by Th\'eor\`eme $4.3$ \cite{Bo-1}
$$par degE_*=deg_{\chi}\nu=q_*(c_1^{et}(\nu).\pi^*c_1^{et}\OOO_X(1)^{n-1},$$
where $q:\chi\ra Spec (\CC),$ and apply the equivalence between tensor categories so that
the Gromov-Witten invariants are well defined in the parabolic case.

We at present want to see the relation between the Gromov-Witten invariants defined above
and the one known in the case of the Kontsevich moduli stack $\KKK_{g,n}(\chi,\beta)$ of stable twisted maps.

\begin{proposition}
Suppose that $E$ is a general stable bundle on the stack $\HHH$, then we can identify the Kontsevich moduli stack 
of the twisted stable maps with the smooth part of the open substack of the Quot-stack.
$$\KKK_{g,n}(Grass(\HHH,\beta_e))\simeq (Quot_0^{k'_n,e'_n})^s(\HHH)\times\CCC^n$$ and
$$\KKK_{g,n}(Grass(\BBB,\beta_f))\simeq (Quot_0^{k'_2,f'_n})^s(\BBB)\times\CCC^n.$$
\end{proposition}
\begin{proof}
We remind that $$\KKK_{g,n}(Grass(\HHH,\beta_e))$$ consists of twisted stable maps $f$ from $\CCC$ to $Grass(\HHH,\beta_e)$
such that the latter is a representable map, the markings are twisted curves and the automorphism group fixing
them is finite, in other words, in terms of their respective coarse moduli spaces, the induced
map is stable. Moreover, $f_*([\CCC])=\beta_e\in H^2((Grass(\HHH,\QQ))$. This stack comes with
evaluations maps $$ev_i:\KKK_{g,n}(Grass(\HHH,\beta_e))\mapsto Grass(\HHH),$$ sending the i-th marking to its image.
When $E$ is general stable bundle on the stack $\HHH$, then one uses the fact that $\KKK_{g,n}(Grass(\HHH,\beta_e))$ (see Example $5.4$ \cite{P-R}) is a pure stack whose every component is generically smooth of expected dimension and general element in every connected
component corresponds to a vector bundle quotient. Hence it follows from Lemma $B.9.2$ of \cite{Fu} that if we choose
general translates of the Schubert cells whose codimensions add up to the expected dimension of the Kontsevich stack
of the twisted stable maps then their intersections after pulling back by the various relevant evaluation maps can be 
realized with the smooth part of the stack $Quot_0^{k,e'_n}(\HHH)\times\CCC^n$, where $Quot_0^{k,e}$ is an open substack of 
$Quot^{k,e}.$
\end{proof}

\section{Case of deformation theory}
We want thereafter to construct the Gromov-Witten invariants for parabolic vector bundles over an orbifold locally
in using deformation theory.

We first construct the versal deformation of $(t,\lambda)$ parabolic connections
to deduce the Kuranishi space of parabolic bundles.

We set
\[
 T_n :=\left.\left\{ (t_1,\ldots,t_n)\in\overbrace{X\times\cdots\times X}^n
 \right| \text{$t_i\neq t_j$ for $i\neq j$} \right\}
\]
for a positive integer $n$.
For integers $d,r$ with $r>0$, we set
\[
 \Lambda^{(n)}_r(d):=
 \left\{ (\lambda^{(i)}_j)^{1\leq i\leq n}_{0\leq j\leq r-1}
 \in \CC^{nr}
 \left|
 d+\sum_{i,j}\lambda^{(i)}_j=0
 \right\}\right..
\]
Take an element $t=(t_1,\ldots,t_n)\in T_n$ and
$\lambda=(\lambda^{(i)}_j)_{1\leq i\leq n,0\leq j\leq r-1}
\in\Lambda^{(n)}_r(d)$.
\begin{definition}\rm $(E,\nabla,\{l^{(i)}_*\}_{1\leq i\leq n})$ is said to be
a $(t,\lambda)$-parabolic connection of rank $r$ if \\
$(1)$ $E$ is a rank $r$ algebraic vector bundle on $X$, and \\
$(2)$ $\nabla: E \ra E\otimes\Omega_C^1 (\log(t_1+\dots+t_n)$ is a connection, and\\
$(3)$ for each $t_i$, $l^{(i)}_*$ is a filtration of $E|_{t_i}=l^{(i)}_0\supset l^{(i)}_1
 \supset\cdots\supset l^{(i)}_{r-1}\supset l^{(i)}_r=0$
 such that $\dim(l^{(i)}_j/l^{(i)}_{j+1})=1$ and
 $(\Res_{t_i}(\nabla)-\lambda^{(i)}_j\mathrm{id}_{E|_{t_i}})
 (l^{(i)}_j)\subset l^{(i)}_{j+1}$
 for $j=0,\ldots,r-1$.
\end{definition}
\begin{remark}\rm
By condition (3) above and \cite {EV-1}, we have
\[
 \deg E=\deg(\det(E))=-\sum_{i=1}^n\Tr\Res_{t_i}(\nabla)
 =-\sum_{i=1}^n\sum_{j=0}^{r-1}\lambda^{(i)}_j=d.
\]
\end{remark}

Let $T$ be a smooth algebraic scheme which is a covering 
of the moduli stack of $n$-pointed smooth projective curves of genus $g$
over $\CC$ and take a universal family $(\CCC,\tilde{t}_1,\ldots,\tilde{t}_n)$
over $T$.

\begin{definition}\rm
We denote the pull-back of $\CCC$ and $\tilde{t}$ with respect to the morphism
$T\times\Lambda^{(n)}_r(d)\rightarrow T$
by the same characters $\CCC$ and
$\tilde{t}=(\tilde{t}_1,\ldots,\tilde{t}_n)$.
Then $D(\tilde{t}):=\tilde{t}_1+\cdots+\tilde{t}_n$
becomes a family of  Cartier divisors on $\CCC$ flat over
$T\times\Lambda^{(n)}_r(d)$.
We also denote by $\tilde{\lambda}$ the pull-back of the
universal family on $\Lambda^{(n)}_r(d)$ by the morphism
$T\times\Lambda^{(n)}_r(d)\rightarrow \Lambda^{(n)}_r(d)$.
We define a functor
$\MMM^{\boldsymbol{\alpha}}_{\CCC/T}
(\tilde{t},r,d)$
from the category of locally noetherian schemes over
$T\times\Lambda^{(n)}_r(d)$ to the category of sets by
\[
 \MMM^{\boldsymbol{\alpha}}_{\CCC/T}
(\tilde{t},r,d)(S):=
 \left\{ (E,\nabla,\{l^{(i)}_j\}) \right\}/\sim,
\] where
\begin{enumerate}
\item $E$ is a vector bundle on $\CCC_S=\CCC\times_{T\times\Lambda^{(n)}_r(d)} S$ of rank $r$,
\item $\nabla:E\rightarrow E\otimes\Omega^1_{\CCC_S/S}(D(\tilde{t})_S)$
 is a relative connection,
\item $E|_{(\tilde{t}_i)_S}=l^{(i)}_0\supset l^{(i)}_1
 \supset\cdots\supset l^{(i)}_{r-1}\supset l^{(i)}_r=0$
 is a filtration by subbundles such that
 $(\Res_{(\tilde{t}_i)_S}(\nabla)-(\tilde{\lambda}^{(i)}_j)_S)(l^{(i)}_j)
 \subset l^{(i)}_{j+1}$
 for $0\leq j\leq r-1$, $i=1,\ldots,n$,
\item for any geometric point $s\in S$,
$\dim (l^{(i)}_j/l^{(i)}_{j+1})\otimes k(s)=1$ for any $i,j$ and
$(E,\nabla,\{l^{(i)}_j\})\otimes k(s)$
is $\alpha$-stable.
\end{enumerate}
Here $(E,\nabla,\{l^{(i)}_j\})\sim
(E',\nabla',\{l'^{(i)}_j\})$ if
there exist a line bundle $\LLL$ on $S$ and
an isomorphism $\sigma:E\stackrel{\sim}\ra E'\otimes\LLL$ 
such that $\sigma|_{t_i}(l^{(i)}_j)=l'^{(i)}_j$ for any $i,j$ and the diagram
\[
 \begin{CD}
  E @>\nabla>> E\otimes\Omega^1_{\CCC/T}(D(\tilde{t})) \\
  @V \sigma VV  @V\sigma\otimes\mathrm{id} VV \\
  E'\otimes\LLL @>\nabla'>>
  E'\otimes\Omega^1_{\CCC/T}(D(\tilde{t}))\otimes\LLL \\
 \end{CD}
\]
commutes.
\end{definition}
We now can construct the moduli space of this functor.
\begin{theorem}\label{moduli-exists}
There exists a relative fine moduli scheme
\[
 M^{\boldsymbol{\alpha}}_{\CCC/T}(\tilde{t},r,d)\rightarrow T\times\Lambda^{(n)}_r(d)
\]
of $\boldsymbol{\alpha}$-stable parabolic connections of rank $r$ and degree $d$,
which is smooth, irreducible and quasi-projective and has an algebraic symplectic structure.
The fiber $M^{\boldsymbol{\alpha}}_{\CCC_x}(\tilde{t}_x,\lambda)$
over $(x,\lambda)\in T\times\Lambda^{(n)}_r(d)$ is the irreducible moduli space of
$\boldsymbol{\alpha}$-stable $(\tilde{t}_x,\lambda)$ parabolic connections
whose dimension is $2r^2(g-1)+nr(r-1)+2$ if it is non-empty.
\end{theorem}
\begin{proof}
See \cite{I}.
\end{proof}

Let $(\tilde{E},\tilde{\nabla},\{\tilde{l}^{(i)}_j\})$
be a universal family on
$\CCC\times_T M^{\boldsymbol{\alpha}}_{\CCC/T}(\tilde{t},r,d)$.
We define a complex $\GGG^{\fatdot}$ by
\begin{align*}
 \GGG^0&:=\left\{ s\in {\mathcal End}(\tilde{E}) \left|
 \text{$s|_{\tilde{t}_i\times M^{\boldsymbol{\alpha}}_{\CCC/T}(\tilde{t},r,d)}
 (\tilde{l}^{(i)}_j)\subset\tilde{l}^{(i)}_j$ for any $i,j$}
 \right\}\right. \\
 \GGG^1&:=\left\{ s\in {\mathcal End}(\tilde{E})
 \otimes\Omega^1_{\CCC/T}(D(\tilde{t})) \left|
 \text{$\Res_{\tilde{t}_i\times M^{\boldsymbol{\alpha}}_{\CCC/T}(\tilde{t},r,d)}(s)
 (\tilde{l}^{(i)}_j)\subset \tilde{l}^{(i)}_{j+1}$ for any $i,j$}
 \right\}\right. \\
 \nabla_{\GGG^{\fatdot}} &: \GGG^0 \lra \GGG^1 ; \quad
 \nabla_{\GGG^{\fatdot}}(s)=\tilde{\nabla}\circ s - s\circ\tilde{\nabla}.
\end{align*}

As in the previous section, we can construct the Kuranishi space of $(t,\lambda)$-parabolic connections on a smooth projective curve
in using the hypercohomology of $\GGG^{\fatdot}$.
\begin{theorem}
Let $X$ be a smooth projective curve over $k$, $(\EEE,\nabla,\{l^{(i)}_{*}\})$ a $(t,\lambda)$-parabolic connection on $X$, $\GGG^{\fatdot}$ the complex of sheaves on $X$ defined above, $W=\HH^{1}(X,\GGG^{\fatdot})$, $(\delta_1\dots,\delta_N)$ a basis of $W$ and $(t_1,\dots,t_N)$ the dual coordinates on $W$.
Let $W_k$ denote the $k$-th infinitesimal neighborhood of $0$ in $W$, and $(\EEE_1,\nabla_1,\{l^{(i)}_{*}\}_{1})$ the universal first order
deformation of $(\EEE,\nabla,\{l^{(i)}_{*}\})$ over $X\times W_1$ in the class of $(t,\lambda)$-parabolic connections. Then there exists a formal power series
$$f(t_1,\dots,t_N)=\sum_{k=2}^{\infty} f_{k}(t_1\dots,t_N)\in \HH^{2}(X, \GGG^{\fatdot})[[t_1,\dots,t_N]],$$
where $f_k$ is  homogeneous of degree $k$ ($k\geq 2$), with the following property.
Let $I$ be the ideal of $k[[t_1,\dots,t_N]]$ generated by the image of the map $f^*:\HH^{2}(X,\GGG^{\fatdot})\rar k[[t_1,\dots,t_N]]$, adjoint to $f$.
Then for any $k\geq 2$, the triple $(\EEE_1,\nabla_1,\{l^{(i)}_{*}\}_{1})$ extends to a $(t,\lambda)$-parabolic connection $(\EEE_k,\nabla_k,\{l^{(i)}_{*}\}_{k})$ on $X\times V_k$, where
$V_k$ is the closed subscheme of $W_k$ defined by the ideal $I\otimes k[[t_1,\dots,t_N]]/(t_1,\dots,t_N)^{k+1}.$
\end{theorem}
\begin{proof}
This follows of the proof by construction in Theorem $3.6$ of \cite{Machu}.
\end{proof}
We now want to construct the Kuranishi space of $T$-parabolic bundles.
Let $T$ be a finite set of smooth points $\{P_1,\dots,P_n\}$ of $X$ and $W$ a vector bundle on $X$.
\begin{definition}
By a quasi-parabolic structure on a vector bundle $W$ at a smooth point $P$ of $X$, we mean a choice of a flag
$$W_P=F_1(W)_P\supset F_2(W)_P\supset...\supset F_l(W)_P=0,$$ in the fibre $W_P$ of $W$ at $P$.
A parabolic structure at $P$ is a pair consisting of a flag as above and a sequence $0\leq\alpha_1<\alpha_2<...<\alpha_l<1$ of weights of $W$ at $P$.
\end{definition}
The integers $k_1=\dim F_1(W)_P-\dim F_2(W)_P$,\ldots, $k_l=\dim(F_l(W)_P)$ are called the multiplicities of $\alpha_1,\dots,\alpha_l$.
A $T$-parabolic structure on $W$ is the triple consisting of a flag at $P$, some weights $\alpha_i$, and their multiplicities $k_i$.
A vector bundle $W$ endowed with a $T$-parabolic structure is called a $T$-parabolic bundle. 
\begin{definition}
A $T$-parabolic bundle $W_1$ on $X$ is a $T$-parabolic subbundle of a $T$-parabolic bundle $W_2$ on $X$, if
$W_1$ is a subbundle of $W_2$ and at each smooth point $P$ of $T$, the weights of $W_1$ are a subset of those of $W_2$.
Further, if we take the weight $\alpha_{j_0}$ such that $1\leq j_0\leq m$, and  the weight $\beta_{k_0}$ for the greatest integer $k_0$ such that $F_{j_0}(W_1)_P\subset F_{k_0}(W_2)_P$, then $\alpha_{j_0}=\beta_{k_0}$.
\end{definition}
\begin{definition}
The parabolic degree of a $T$-parabolic vector bundle $W$ on $X$ is 
$$\mathrm{par\deg}(W):=\deg(W)+\sum_{P\in I}\sum_{i=1} ^{r}k_i (P)\alpha_i (P).$$
\end{definition}
\begin{definition}
A $T$-parabolic bundle $W$
is stable (resp. semistable) if
for any proper nonzero $T$-parabolic subbundle $W'\subset W$
the inequality
\begin{gather*} {\mathrm{par\deg} W'}<{(\textrm{resp. $\leq$})}
 \frac{\mathrm{par\deg W} \rk(W')}{\rk W}\end{gather*}
holds.
\end{definition}

We have a forgetful map $g$ from $(t,\lambda)$ parabolic connections to $T$-parabolic bundles.
We thus can construct the Kuranishi space of $T$-parabolic bundles by following an analogous argument
to the one given above. We first introduce the Higgs field $\Phi:\EEE\rar\EEE\otimes\Omega_X^1(D)$ defined as follows:
$$\forall p\in X, \forall f\in\OOO_{X,p}, \forall s\in\EEE_P, \Phi(fs)=f\Phi(s).$$
We afterwards consider a parabolic bundle $\EEE$ with fixed weights and parabolic points $P_1,\dots,P_N$. We set $L=K\otimes\OOO(P_1,\dots,P_N)$, the line bundle associated to the canonical divisor together with the divisor of poles $D=P_1+\dots+P_N$.
The sheaf of rational $1$-forms on $X$ is identified with the sheaf of rational sections of the canonical bundle having single poles at points
$P_1,\dots,P_N$. We replace $t_i$ by $P_i$, for $i=1,\dots,N$ and $M^{\boldsymbol{\alpha}}_{\CCC/T}(\tilde{t},r,d)$ by $M^s_{T}$. 
We define a complex $\BBB^{\fatdot}$ by
\begin{align*}
 \BBB^0&:=\left\{ s\in {\mathcal End}(\tilde{E}) \left|
 \text{$s|_{\tilde{P}_i\times M^{s}_{Z,\CCC/T}(\tilde{P},r,d)}
 (\tilde{l}^{(i)}_j)\subset\tilde{l}^{(i)}_j$ for any $i,j$}
 \right\}\right. \\
 \BBB^1&:=\left\{ s\in {\mathcal End}(\tilde{E})
 \otimes\Omega^1_{\CCC/T}(D(\tilde{Pi})) \left|
 \text{$\Res_{\tilde{P}_i\times M^{s}_{Z,\CCC/T}(\tilde{P},r,d)}(s)
 (\tilde{l}^{(i)}_j)\subset \tilde{l}^{(i)}_{j+1}$ for any $i,j$}
 \right\}\right.\\
 \ad\Phi_{\BBB^{\fatdot}} &: \BBB^0 \lra \BBB^1 ; \quad
 \ad\Phi_{\BBB^{\fatdot}}(s)=\tilde{\Phi}\circ s - s\circ\tilde{\Phi}.
\end{align*}
From this, we deduce the construction of the Kuranishi space of $T$-parabolic bundles on a smooth projective curve.
\begin{theorem}
Let $X$ be a smooth projective curve over $k$ or a complex space (in which case $k=\CC$), $\EEE$  a $T$-parabolic bundle on $X$, $\BBB^{\fatdot}$ the complex of sheaves on $X$ defined as above, $W=\HH^{1}(X,\BBB^{\fatdot})$, $(\delta_1\dots,\delta_N)$ a basis of $W$ and $(t_1,\dots,t_N)$ the dual coordinates on $W$.
Let $W_k$ denote the $k$-th infinitesimal neighborhood of $0$ in $W$, and $\EEE_1$ the universal first order
deformation of $\EEE$ over $X\times W_1$. Then there exists a formal power series
$$f(t_1,\dots,t_N)=\sum_{k=2}^{\infty} f_{k}(t_1\dots,t_N)\in \HH^{2}(X, \BBB^{\fatdot})[[t_1,\dots,t_N]],$$
where $f_k$ is  homogeneous of degree $k$ ($k\geq 2$), with the following property.
Let $I$ be the ideal of $k[[t_1,\dots,t_N]]$ generated by the image of the map $f^*:\HH^{2}(X,\BBB^{\fatdot})^{*}\rar k[[t_1,\dots,t_N]]$, adjoint to $f$.
Then for any $k\geq 2$, $\EEE_1$ extends to a $T$-parabolic bundle $\EEE_k$ on $X\times V_k$, where
$V_k$ is the closed subscheme of $W_k$ defined by the ideal $I\otimes k[[t_1,\dots,t_N]]/(t_1,\dots,t_N)^{k+1}.$
\end{theorem} 
\begin{proof}
This follows of the proof by construction in Theorem $3.6$ of \cite{Machu}.
\end{proof}

\begin{definition}
The inverse limit $\mathbb{V}=\underleftarrow{\lim}V_k$ is called the formal Kuranishi space of $\EEE$, and $\boldsymbol{\EEE}=\underleftarrow{\lim} \EEE_k$ the formal universal parabolic bundle over $\mathbb{V}$.
\end{definition}

We can hence apply the previous method of constructing locally the Gromov-Witten invariants of parabolic bundles over an orbifold,
where we replace $\HHH=r\sqrt{(C,p_1)}\times...r\sqrt{(C,p_n)})$
by $\HHH=r\sqrt{(\mathbb{V},p_1)}\times...r\sqrt{(\mathbb{V},p_n)})$, idem
for $\BBB=r\sqrt{(C,E_1)}\times r\sqrt{(C,E_2)}$.
\begin{definition}
For a pair $(\HHH,e')$ over $C$ and $(\BBB,f')$, we define the Gromov-Witten invariants $N_{n,e'}P(X_1,...,X_{k_{n}})$ as
$$N_{n,e'}P(X_1,...,X_{k'_{n}})=\sum_{k_1,...,k_n}P(c_1(\check{\FFF}),...,c_{k_1+...+k_n}(\check{\FFF}))\cap[Quot_{vir}^{(k'_n,e')}(\HHH)],$$
with $k'_n=\sum_{1}^{n}k_i$ 
and $$N_{n',f'}P(X_1,...,X_{k'_{2}})=\sum_{k_1,...,k_2}P(c_1(\check{\FFF}),...,c_{k_1+k_2}(\check{\FFF}))\cap[Quot_{vir}^{(k'_2),f')}(\BBB)].$$
\end{definition}

\section{Generalization}
It is then natural to ask what happens in the framework of the generalization of the primitive definition
of a parabolic structure at a marked point.

We at present consider a connected complex reductive algebraic group $G$ containing a simply-connected and simple compact group $K$ such
$T$ its maximal torus in $K$ and $P$ a parabolic subgroup of $G$. 
We denote $W$ the Weyl group and $W_{P}$ its subgroup generated by the simple reflection of roots of the Levi subgroup of P.
We also denote $\boldsymbol{t}$ the Cartan subalgebra containing $\boldsymbol{t_+}$ the positive Weyl chamber and $\alpha_{0}$ the highest
root. Let $\pi:E\ra C$ be a principal $G$-bundle over $C$ with marked points $p_1,...,p_n$.
\begin{definition}
A parabolic structure at $p_i$ consists of the following data:\\
$(1)$ a standard parabolic subgroup $P_i\subset G$.\\
$(2)$ $\phi_i\in E_x/P_i$ of the reduction of the fiber $E_x$ to $P_i$.\\
$(3)$ a marking $\mu_i\in\UUU$, where $\UUU=\{\epsilon\in\boldsymbol{t_+}\mid \alpha_{0}(\epsilon)\leq 1\}$ with $\alpha_0(\mu_i)<1$, where 
the stabilizer $G_{\mu_{i}}$ under the adjoint action is a Levi subgroup of $P_i$.
\end{definition}
Hence a parabolic bundle on $(C,p_1,...,p_n)$ is a bundle $E$ with parabolic structure at these points.
\begin{definition}
A reduction of structure group of $E$ at $P$ is a map
$$\sigma:C\ra E/P.$$
\end{definition}
Note that for any $\lambda\in\Lambda_P$, where $\Lambda_P$ is the character subgroup of P, $\sigma^{*}(E(\lambda))$ is 
a line bundle on $C$ whose degree is in $\ZZ$. The latter will be used to define the root stacks as previously.

One of the goal of this paper is to determine the number of maximal parabolic subbundles of a sufficiently general stable bundle.
So, we make reference to the definition of Ramanathan for the semistability.
\begin{definition}
E semistable if $\deg(\sigma^{*}(E(\lambda))\leq 0, \forall\lambda\in\Lambda_{P,+}.$
\end{definition}
Hence, we see that the definition of semistability for parabolic principal
$G$-bundles depend on the relative position of $\sigma$ and $\phi_î$.
Given two parabolic subgroups $P'_1=\Ad(g)P_1$, $P'_2=\Ad(g)P_2$, define
their relative position $(P'_1,P'_2)\in WP_1\setminus{\ W/WP_2}$ to be the image
of $(g_1,g_2)$ under the map $$G\times G\ra WP_1\setminus{\ W/WP_2}.$$
We deduce the following definition.
\begin{definition}
$E$ is stable (resp. semistable) if for any maximal subgroup $P$ of $G$ and $\sigma$, we have
$$\deg(\sigma^{*}(E(\lambda))+\sum_{i=1}^{p}\omega_P(w_i\mu_i)< 0 (resp.\leq 0),$$ where $w_i=(\phi_i,\sigma(p_i))$ and
$\omega_p$ fundamental weights of $P$.\\ 
A sufficiently general stable bundle is a bundle whose associated parameter space
of the stable bundles is a dense open subset of a Zariski open subset of the parameter space of bundles.
\end{definition}

We apply our previous results to this case and get the definition of the Gromov-Witten invariants in the
case of parabolic bundles where we replace $\HHH=r\sqrt{(C,p_1)}\times...r\sqrt{(C,p_n)})$
by $\HHH=r\sqrt{(\sigma^{*}(E(\lambda)),p_1)}\times...r\sqrt{(\sigma^{*}(E(\lambda)),p_n)})$, idem
for $\BBB=r\sqrt{(C,E_1)}\times r\sqrt{(C,E_2)}$.
Finally, we arrive at the following definition:
\begin{definition}
For a pair $(\HHH,e')$ over $C$ and $(\BBB,f')$, we define the Gromov-Witten invariants $N_{n,e'}P(X_1,...,X_{k_{n}})$ as
$$N_{n,e'}P(X_1,...,X_{k'_{n}})=\sum_{k_1,...,k_n}P(c_1(\check{\FFF}),...,c_{k_1+...+k_n}(\check{\FFF}))\cap[Quot_{vir}^{(k'_n,e')}(\HHH)],$$
with $k'_n=\sum_{1}^{n}k_i$ 
and $$N_{n',f'}P(X_1,...,X_{k'_{2}})=\sum_{k_1,...,k_2}P(c_1(\check{\FFF}),...,c_{k_1+k_2}(\check{\FFF}))\cap[Quot_{vir}^{(k'_2),f')}(\BBB)].$$
\end{definition}

We show that the Gromov-Witten invariants defined hence are independent on the choice of the orbifold curve $C$ of genus $g$.
For this, we first note that for a collection of root stacks over a genus-$g$ curve, the existence of 
a smooth irreducible variety $T$ and a family $\FFF$ of root stacks on $C\times T$ whose restriction at the fiber over a point in $T$ gives a root stack, constructed in using the universal property arising from the smoothness and the irreducibility of the moduli stack of roots.
Therefore, we obtain for a such family on $C\times T$.

\begin{proposition}
If $\FFF$ is a family of root stacks on $C\times T$, with $T$ a smooth curve, and $e$ being chosen
such that $s_e\geq s_{\FFF}$, then the Gromov-Witten invariants are independent of the choice of points $x\in B$.
\end{proposition}

\begin{proof}
We show that the relative Quot-scheme $e:\Quot(\FFF)\ra T$ is a locally complete intersection morphism, and in 
particular flat. For this, the hypothesis $s_e\geq s_{\FFF}$ enables to say that $\Quot(\FFF_x)$ is generically smooth of expected dimension. Therefore, the proposition follows from Lemma $1.6$ of \cite{Ber}.
\end{proof}

\begin{lemma}\label{line bundle}
Let $\HHH$ be the root stack of multidegree $e$ and of rank $n$ and $L$ a line bundle of degree $d$.
Then we have $s(\HHH)=s(\HHH\otimes L)$, and the Gromov Witten invariants of $\HHH$ and $\HHH\otimes L$ are
related by the following formula
$$N_{e'+nd, e'+k_nd}(P(X_1,...,X_{k_n}),\HHH\otimes L)=N_{d,e'}(P(X_1,...,X_{k_n}),\HHH).$$
\end{lemma}
\begin{proof}
This follows arising immediately from the isomorphism between the Quot stacks $\Quot^{k_n,e'+k_nd}(\HHH\otimes L)$ and $\Quot^{k_n,e'}(\HHH)$.
\end{proof}

Before continuing our study in the case of general stable parabolic bundles, we make an another approach
to refine the construction of the Gromov-Witten invariants for parabolic bundles in using the correspondence
between the equivariant bundles and the parabolic bundles.

We start with a cyclic group $\Gamma$ of order $N$ acting on a curve $\tilde{X}$ with
quotient $X=\tilde{X}/\Gamma$ with a map $\pi:\tilde{X}\ra X,$ ramified at the $x_i$.
Fix $\mu_1,...,\mu_n\in\UUU$ with $e^{\mu_i}=1,\forall 1\leq i \leq n.$
Mehta-Seshadri proved that there is a one-to-one correspondence between the
set of isomorphism classes of $\Gamma$-equivariant bundles $\tilde{E}$ on $\tilde{X}\times S$ with
a $\Gamma$-action on the fibers $E_{x_i}$ lie in the conjugation class of $e^{u_i}$ with
the set of isomorphism classes of parabolic bundles $E$ on $X\times S$.

It is not difficult to pass from the construction of a equivariant bundle to a parabolic bundle and conversely
where the parabolic structure is given by the filtration at the ramified points by order of the vanishing.
To recover the parabolic bundle , we quotient by the group $\Gamma$ and use the transition
functions $z^{-N\mu_i/2i\pi}$.

We apply our previous results to this case and get the definition of the Gromov-Witten invariants in the
case of parabolic bundles.
Finally, we arrive at the following definition:
\begin{definition}
For a pair $(\HHH,e')$ over $C$ and $(\BBB,f')$, we define the Gromov-Witten invariants $N_{n,e'}P(X_1,...,X_{k_{n}})$ as
$$N_{n,e'}P(X_1,...,X_{k'_{n}})=\sum_{k_1,...,k_n}P(c_1(\check{\FFF}^{-N\mu_i}/\Gamma),...,c_{k_1+...+k_n}(\check{\FFF}^{-N\mu_i}/\Gamma)\cap[Quot_{vir}^{(k'_n,e')}(\HHH)],$$
with $k'_n=\sum_{1}^{n}k_i$ 
and $$N_{n',f'}P(X_1,...,X_{k'_{2}})=\sum_{k_1,...,k_2}P(c_1(\check{\FFF}^{-N\mu_i}/\Gamma),...,c_{k_1+k_2}(\check{\FFF}^{-N\mu_i}/\Gamma)\cap[Quot_{vir}^{(k'_2),f')}(\BBB)].$$
\end{definition}

\section{general stable parabolic bundles}
We now assume that we work with the moduli stacks $Quot^{(k,e')}[\HHH]$ and $[Quot^{(k,f')}[\BBB])$
of stable objects. In the case of the stability of the objects, we can define the notion
of $s$-invariant as follows.
$$s_{k_n}(\HHH)=k_n(n-k_n)(g-1)+\epsilon, 1\leq\epsilon\leq n-1, s_{k_{n+2}}(\BBB)=k_{n+2}(2-k_{n+2})(g-1)+1.$$
Let $e_{max,d}$ be the degree of the maximal subbundle of a general stable bundle of degree $d$.
We can also define those in the refinement of our definition for the Gromov-Witten invariants.
$$s_{k_n}(\HHH)=k_n(n-k_n)(g-1)+\epsilon+N\sum_{i=1}^{n}\mu_i, 1\leq\epsilon\leq n-1, s_{k_{n+2}}(\BBB)=k_{n+2}(2-k_{n+2})(g-1)+1+N\sum_{i=1}^{2}\mu_i.$$

\begin{proposition}
The moduli stack of roots admits an open moduli stack $\UUU$ with the property that for each
$H\in\UUU$, and for each $e'\leq e_{max,n}$ (resp. $e'\leq e_{max,2}$, every component of 
the quot stack $\Quot^{k_n,e'}(H)$ (resp. $\Quot^{k_{n+2},f'}(H)$) is smooth of expected dimension
and satisfies the property that general elements in every irreducible component
correspond to root substacks of $H$.
\end{proposition}

\begin{proof}
The proof of the Proposition relies on the torsion free part
of the Quot-stack is generically smooth from Proposition $6.7$ of \cite{Ho-1} and
on the contradiction on the dimensions show that an irreducible
component of the Quot-stack is torsion free
\end{proof}

We search for some relations between the Gromov Witten invariants for parabolic bundles.

We have the following explicit formula for the Gromov-Witten invariants. 
\begin{theorem}\label{computation}
Let $n$ and a multiinteger $d$ be fixed. Set $d=an-b$, where $0\leq b<n$ and
$e\leq e_{max}(d)$. Let $P(X_1,..., X_{k_n})$ be a polynomial of weighted degree
$$dk_n'-ne_1+k_n'(n-k_n')(1-g).$$
Then we deduce the following relation
$$N_{d,e'}(P(X_1,...,X_{k'_n}))=N_{0,e'-ak'_n}(X_{k'_n}^{b}P(X_1,...,X_{k'_n})).$$
\end{theorem}

\begin{proof}
This follows of the previous Proposition and Lemma \ref{line bundle}.
\end{proof}

The above result is independent of the choice of the orbifold
curve $C$ of genus $g$.
We at present remind the formula of Vafa and Intriligator, proved by A. Bertram (see in \cite{Ber}, \cite{Ber-Das-Went} updated to our case
for an explicit computation of Gromov Witten invariants $N_{0,e'}(P(X_1,...,X_{k^n}).$
Let $P(X_1,...,X_{k'_n})=\prod_{i=1}^{m}X_{a_i}$ be a polynomial with $0<a_i\leq k'_n$ such that the
weighted degree of $P$ is $\sum_{i} (k'_n-a_i+1)=-e'n+k'_n(n-k'_n)(1-g).$ Then we have the following.
\begin{proposition}\label{intrigilator}
For the polynomial $P=\prod_{i=1}^{m}X_{a_i}$, defined as above, the Gromov Witten invariants are constructed as follows.
We introduce a few notation $k'=k'_n$, $\alpha=k'(g-1), \beta=(-1)^{e'(k'-1)+(g-1)k'(k'-1)/2}$, and
$S=\{(\rho_1,..,\rho_{k'})\mid\rho_i^{n}=1,\rho_i\neq\rho_j\}$ and $\Delta=\prod_{i=1}^{m}\sigma_{k'-a(l)+1}(\rho)$ to get
$$\frac{n^{\alpha\beta}}{k'!}\sum_{S}\frac{\Delta}{(\prod_{i=1}^{n}\rho_i\prod_{i\neq j}(\rho_i-\rho_j))^{g-1}},$$
where $\sigma_j(\rho)$ is the $j$-symmetric polynomial in $\rho_i$'s.
\end{proposition}

\section{maximal parabolic subbundles}
We want to provide some examples of the computation of the Gromov-Witten invariants for parabolic bundles
over an orbifold $C$ of genus $g$, in particular for the number of maximal subbundles of a sufficiently parabolic stable
bundle denoted $m(n,d,k'_n,g)$ and $m(2,d,k'_2,g)$ in certain cases.
We first state the following propostion.
\begin{proposition}
For a general root stack $\EEE$, the Quot-stack $\Quot^{k'_n,e'_{max,d}}(\EEE)$ is 
a zero-dimensional smooth stack.
\end{proposition}
\begin{proof}
We are in the case where $e'=e'_{max,d}$, hence in using the result of Mukai and Sakai \cite{M-S}, with
Lemma \ref{line bundle}, we deduce the result.
\end{proof}  
Futhermore, we can count the number of points $m(n,d,k'_n,g)$ (resp. $m(2,d,k'_2,g)$)
in the Quot stack. We have the following explicit formula for the number $m(n,d,k'_n,g)$.
\begin{theorem}
In using Theorem \ref{computation}, we get with $\beta=(-1)^{(k'-1)(bk'-(g-1)k'^2/n}$, 
$$\frac{n^{\alpha\beta}}{k'!}\sum_{S}\frac{\Delta^{b-g+1}}{(\prod_{i\neq j}(\rho_i-\rho_j))^{g-1}}.$$
\end{theorem}
\begin{proof}
Use Proposition \ref{intrigilator} and Theorem \ref{computation}.
\end{proof}

We deduce the following Corollary.
\begin{corollary}
$m(n,d,1,g)=n^{ng}$, and $m(2,d,1,g)=n^{2g}$.
\end{corollary}

\renewcommand\refname{References}

\end{document}